\documentclass[11pt,a4paper,reqno]{amsart}
\usepackage{amsfonts,amssymb,amsmath,amsthm,graphicx,color,amscd,xspace,verbatim,dsfont}
\usepackage{scalerel}
\makeatletter

\makeatletter
\def\@tocline#1#2#3#4#5#6#7{\relax
  \ifnum #1>\c@tocdepth 
  \else
    \par \addpenalty\@secpenalty\addvspace{#2}%
    \begingroup \hyphenpenalty\@M
    \@ifempty{#4}{%
      \@tempdima\csname r@tocindent\number#1\endcsname\relax
    }{%
      \@tempdima#4\relax
    }%
    \parindent\z@ \leftskip#3\relax \advance\leftskip\@tempdima\relax
    \rightskip\@pnumwidth plus4em \parfillskip-\@pnumwidth
    #5\leavevmode\hskip-\@tempdima
      \ifcase #1
       \or\or \hskip 1em \or \hskip 2em \else \hskip 3em \fi%
      #6\nobreak\relax
      \dotfill
      \hbox to\@pnumwidth{\@tocpagenum{#7}}
    \par
    \nobreak
    \endgroup
  \fi}
\makeatother

\newtheorem{theorem}{Theorem}[section]
\newtheorem{lemma}[theorem]{Lemma}
\newtheorem{proposition}[theorem]{Proposition}

\theoremstyle{definition}

\newcommand{\R}{{\mathbb R}}

\newcommand{\beqn}{\begin{eqnarray}}
\newcommand{\eeqn}{\end{eqnarray}}   
\newcommand{\beq}{\begin{eqnarray*}}
\newcommand{\eeq}{\end{eqnarray*}}

\newcommand{\bea}{\begin{eqnarray}}
\newcommand{\eea}{\end{eqnarray}}
\newcommand{\be}{\small\begin{equation}}
\newcommand{\bel}[1]{\small\begin{equation}\label{#1}}
\newcommand{\ee}{\end{equation}\normalsize}

\newcommand{\BA}{\begin{array}}
\newcommand{\EA}{\end{array}}
\newcommand{\BAN}{\renewcommand{\arraystretch}{1.2}
\setlength{\arraycolsep}{2pt}\begin{array}}
\newcommand{\BAV}[2]{\renewcommand{\arraystretch}{#1}
\setlength{\arraycolsep}{#2}\begin{array}}

\newcommand{\BSA}{\begin{subarray}}
\newcommand{\ESA}{\end{subarray}}
\newcommand{\BAL}{\begin{aligned}}
\newcommand{\EAL}{\end{aligned}}

   \def\BBN {\mathbb N}    
   \def\BBR {\mathbb R}    \def\BBS {\mathbb S}

\newcommand{\xa}{\alpha}
\newcommand{\xb}{\beta}
\newcommand{\xg}{\gamma}
\newcommand{\xG}{\Gamma}
\newcommand{\xd}{\delta}
\newcommand{\xD}{\Delta}
\newcommand{\xe}{\varepsilon}
\newcommand{\xz}{\zeta}

\newcommand{\xl}{\lambda}
\newcommand{\xL}{\Lambda}
\newcommand{\xm}{\mu}
\newcommand{\xn}{\nu}

\newcommand{\xr}{\rho}
\newcommand{\xs}{\sigma}

\newcommand{\xf}{\phi}
\newcommand{\xF}{\Phi}

\def\bal#1\eal{\small\begin{align*}#1\end{align*}\normalsize}
\def\ba#1\ea{\small\begin{align}#1\end{align}\normalsize}
\numberwithin{equation}{section}

\begin{document}

\title{ Ancient solutions to the Allen Cahn equation in catenoids}
\author[K. T. Gkikas]{Konstantinos T. Gkikas}
\address{Department of Mathematics, University of the Aegean, 
83200 Karlovassi, Samos, Greece\newline
Department of Mathematics, National and Kapodistrian University of Athens, 15784 Athens, Greece
}
\email{kgkikas@aegean.gr}

\date{\today}

\begin{abstract}
Let $N\geq 2$ and $F:\BBR^N\to \BBR $ be the unique increasing radially symmetric function satisfying the minimal surface equation for graphs
with the initial conditions $F(1)=0$ and $\lim_{r\to 1}F_r(r)=\infty;$ $r=|x|.$ We construct an ancient solution  to Allen-Cahn equation $\tilde u_t = \Delta_{M} \tilde u + (1-{\tilde u}^2)\tilde u$ in $M\times(-\infty,0),$ where $M=\{(x, \pm F(|x|)):\;x\in\BBR^N,\;|x|\geq1\}$ is a $N$-dimensional catenoid in $\BBR^{N+1}$ and $\xD_{M}$ is the Laplace Beltrami operator of $M.$ In particular, we construct a solution of the form $u(t,r,F(r))=u(t,r,-F(r))$ such that 

$$
u(t,r,F(r)) \approx   \sum_{j=1}^k (-1)^{j-1}w(r-\rho_j(t))  - \frac 12 (1+ (-1)^{k}) \quad \hbox{ as } t\to -\infty,
$$
where  $w(s)$ is a solution of $w'' + (1-w^2)w=0$  with $w(\pm \infty)= \pm 1,$ given by
$w(s) = \tanh \left(\frac s{\sqrt{2}} \right),$ and 
$$\rho_j(t)=\sqrt{-2(n-1)t}+\frac{1}{\sqrt{2}}\left(j-\frac{k+1}{2}\right)\log\left(\frac {|t|}{\log |t| }\right)+ O(1),\quad j=1,\ldots ,k.$$

\medskip


\medskip


\end{abstract}

\maketitle
\tableofcontents

\section{Introduction}

Let $N\geq2$ . The semilinear parabolic equation 
\ba 
u_t = \Delta u + (1-u^2)u \quad\text{in}\; \BBR\times\BBR^N   \label{ac0} 
\ea
is known as Allen-Cahn equation \cite{ac}. This equation is a typical model for the phase transitions. In particular, representing $u^\pm=\pm1$ the two phases of a material, a solution $u(t,x)$ whose values lie at all times in $[-1,1]$ and in most of the space $\R^N$ takes values close to either $+1$ or $-1,$ corresponds to a continuous realization of the phase state of the material. Such solutions as well as their interfacial region have been a subject of intense study. We recall here that the interfacial region at time $t$ is defined as the set $\BBR^N\setminus\{x\in\BBR^N:u(t,x)\approx\pm1\}$. It is well known that there is a close connection between interfacial region of the solutions and minimal surfaces and surfaces evolving by mean curvature, see for instance \cite{bk,chen,5authors,dkwcpam,evans,ninomiya} and their references.

In the present paper, we construct an ancient solution to the Allen Cahn equation 

\ba 
\tilde u_t = \Delta_M \tilde u + (1-{\tilde u}^2)\tilde u \quad\text{in}\; (-\infty,0)\times M,   \label{accat} 
\ea
where $M$ is a $N$-dimensional catenoid in $\BBR^{N+1}$ and $\xD_M$ is the Laplace Beltrami operator of $M.$ In addition this solution contains one or more transition layers at all negative times. These transition layers are related to the single-transition layer solution of
\bal
w''+w(1-w^2)=0\quad\hbox{in}\;\mathbb{R},\quad\lim_{x\rightarrow\pm\infty}w(x)=\pm1\quad\text{and}\quad w(0)=0,
\eal
which  is given in closed form by
 \be w(x)=\tanh\left(\frac{x}{\sqrt{2}}\right ). \label{w}\ee

In particular, we consider the case that $M$ is given by

\bal
M=\{(x,\pm F(|x|)):\;x\in\BBR^N,\;|x|\geq1\},
\eal
where $F$ is the unique increasing radially symmetric function $F:\BBR^N\to \BBR $ satisfying the minimal surface equation for graphs

\ba\label{minimalsurfeq}
\text{div}\left(\frac{\nabla F}{\sqrt{1+|\nabla F|^2}}\right)&=0\quad\text{in}\;\;\{x\in\BBR^N:|x|>1\},
\ea
with the initial conditions $F(1)=0$ and $\lim_{r\to 1}F_r(r)=\infty;$ $r=|x|.$ In addition, we are concern with ancient solutions of the form $\tilde u(t,Y(r,\theta))=u(t,r),$ where 
$Y(r,\theta):(1,\infty)\times \BBS_{N-1}\to M$ is a local parametrization for $M_{+}:=M\cap\{(x,x_{N+1})\in\BBR^N:\;x_{N+1}>0\}$ in polar coordinates, given by
\bal
Y(r,\theta):=(r\theta,F(r)). 
\eal
Therefore, taking in to account that 

\bal
\xD_{M_+}\tilde u&=\frac{1}{r^{N-1}\sqrt{1+F_r^2}}\frac{d}{dr}\left(\frac{r^{N-1}u_r}{\sqrt{1+F_r^2}}\right)
\eal
and

\ba\label{Fr}
F_r(r)=\frac{1}{\sqrt{r^{2N-2}-1}},
\ea
$u$ should satisfy 

\ba\label{ac2}
u_t=\frac{r^{2(N-1)}-1}{r^{2(N-1)}}u_{rr}+\frac{N-1}{r}u_r+u(1-u^2)\quad\text{in}\;\;(-\infty,0)\times(1,\infty).
\ea

Let us now present the ansatz of our construction. Given $k\geq 1$, we build a solution of the above equation such that
\be
u(t,r)  =  \sum_{i=1}^k   (-1)^{j-1}w ( r - \rho_j(t) )  - \frac 12 ( 1+ (-1)^{k-1})  \ +\ \phi(t,r),
\label{forma} \ee
where $\phi(t,r)$ is a lower order perturbation as $t\to -\infty$ and the functions
\be
\rho_1(t)  < \rho_2(t) < \cdots <  \rho_k(t) 
\label{nn}\ee
satisfy at main order $\rho_j(t) \sim    \sqrt{ -2(n-1) t },$ which is a sphere solution of the equation $\xr'=(N-1)\xr^{-1}$ in $(-\infty,0).$ This equation is at main order the interface dynamic corresponding to a radial symmetric solution of \eqref{ac0},  see \cite{bk} for more details. In our current setting, the dynamics driving the interaction of the different components of the interface for a solution of the form \eqref{forma}
are given at main order by the first-order Toda type system

\be
\frac{1}{\xb}\left(\rho_j'+\frac{n-1}{\rho_j}\right)-e^{-\sqrt{2}(\rho_{j+1}-\rho_j)}+e^{-\sqrt{2}(\rho_{j}-\rho_{j-1})}=0,\quad j=1,\ldots ,k,\;\quad\;t\in(-\infty, 0]\label{toda01},
\ee
with the conventions
$\rho_{k+1}=\infty\quad \hbox{and}\quad \rho_0=-\infty,$ and a explicit constant $\beta>0$. This Toda type system is the same as that in \cite{gklpino1}. In $\BBR^N,$  interaction of interfaces  has already
been considered in \cite{carrpego,chen,gklpino2,fuscohale} for $n=1,$ and in the static higher dimensional setting in \cite{AdW,dkw3,dkw0,dkpw}.

Let us conclude the introduction with our main result.

\begin{theorem} \label{teo2} Given any $k\geq 1$, there exist functions $\rho_j(t)$ as in \eqref{nn}
with
\be {\rho}_j(t)=   \sqrt{-2(n-1)t}  + \frac{1}{\sqrt{2}}\left(j-\frac{k+1}{2}\right)\log\left(\frac{|t|}{\log |t|}\right)+ O(1),\quad j=1,\ldots ,k,\;\; 
\label{forma2}\ee
as $t\to -\infty$,
and an ancient solution $u(t,r)$ of equation \eqref{ac2} of the form \eqref{forma} so that
\bal
\lim_{t\to -\infty} \phi (t,r) = 0 \quad\hbox{uniformly in } r\in[1,\infty) .
\eal
\end{theorem}

\medskip

\noindent \textbf{Acknowledgement.} The author wishes to thank Professor M. del Pino for useful discussions. The research project was supported by the Hellenic Foundation for Research and Innovation (H.F.R.I.) under the  “2nd Call for H.F.R.I. Research Projects to support Post-Doctoral Researchers” (Project Number: 59).

\section{Preliminaries}

Let $u$ be a strong solution of \eqref{ac2}. Setting $y=\sqrt{r^{2(N-1)}-1}$ and $v(t,y)=u(t,(1+y^2)^\frac{1}{2(N-1)}),$  by straightforward calculations, we deduce

\ba\label{chvar}\BAL
\frac{r^{2(N-1)}-1}{r^{2(N-1)}}u_{rr}+\frac{N-1}{r}u_r&=(N-1)^2(1+y^2)^\frac{N-2}{N-1}v_{yy}+(N-1)(2N-3)\frac{y}{(y^2+1)^{\frac{1}{N-1}}}v_y\\
&=\frac{(N-1)^2}{(y^2+1)^{\frac{1}{2(N-1)}}}\frac{d}{dy}\left((y^2+1)^{\frac{2N-3}{2(N-1)}}v_y\right).\EAL
\ea

Hence, it is equivalent to study the equation 

\ba\label{mainequation}
v_t=\frac{(N-1)^2}{(y^2+1)^{\frac{1}{2(N-1)}}}\frac{d}{dy}\left((y^2+1)^{\frac{2N-3}{2(N-1)}}v_y\right)+f(v)\quad\text{in}\quad (-\infty,-T)\times(0,\infty),
\ea
where $f(v)=v(1-v^2)$ and $T>0$ is a constant large enough. 

We list below some notations that we use frequently in the present paper.

\medskip

$\bullet$ Let $w(s)=\tanh(\frac{s}{\sqrt{2}})$ and 
\bal
r(y):=(1+y^2)^\frac{1}{2(N-1)}.
\eal

$\bullet$ Denote by $\eta(t)$ the solution of 
\ba\left\{\BAL
\eta'+\frac{1}{2t}\eta+ \xb e^{-\sqrt{2}\eta}&=0\quad\text{in} \;\;(-\infty,-1]\\
\eta(-1)&=0,
\EAL\right.\label{odesimp21}
\ea
where $\xb$ is defined in \eqref{beta}. In addition, by \cite[Lemma 5.2]{gklpino1}, there exists a positive constant $c_0$ such that

\ba\label{lemodesimpl1}\BAL
-\frac{1}{\sqrt{2}}\log\left(c^{-1}_0\frac{\log |t|}{|t|}\right)&\leq\eta(t)\leq -\frac{1}{\sqrt{2}}\log\left(c_0\frac{\log |t|}{|t|}\right)\quad\forall t\leq-2,
\\
0&\leq-\eta'(t)\leq c_0\frac{\log |t|}{|t|}\quad\forall t\leq-2.
\EAL\ea

$\bullet$ Denote by $\xg_j$ the constants satisfying
\bal-\xg_j=\xg_{k-j+1}=\frac{1}{2}\sum_{i=j}^{k-j}b_i\qquad\mathrm{for}\;j\leq\frac{k}{2},\eal
where
\bal
b_l=-\frac{1}{\sqrt{2}}\log\left(\frac{1}{2\xb}(k-l)l\right)\qquad \text{for any}\;\;l=1,...,k-1.
\eal

$\bullet$ Let
\be \label {fff}
\rho^0_j(t) = \sqrt{-2(n-1)t} + (j-\frac{k+1}{2})\eta +\xg_j\quad\text{for}\;\;j=1,...,k
\ee
and $\xr^0=(\xr_1^0,...,\xr_k^0).$
\subsection{The ansatz}
For notational simplicity, we shall only consider the case of an even
$k$. The odd case can be handled similarly.

We will construct an ancient solution of the form
\bal
v(t,y):=\sum_{j=1}^k(-1)^{j+1}w(r(y)-\xr_j(t))-1+\psi(t,y),
\eal
where

\bal
r(y):=(1+y^2)^\frac{1}{2(N-1)}.
\eal
In addition, we will show that the function $\xr=(\xr_1,...,\xr_k)$ satisfies at main order the Toda type system \eqref{toda01}. 

Throughout the paper, we assume that the function $\rho$ satisfies \eqref{nn} and is given by 

\be\label{forma4}
\rho(t) = \rho^0(t)  + h(t)
\ee
where  the (small) function $h(t)=(h_1(t),...,h_k(t))$ is a parameter to be found, on which we only a priori assume
  
\ba
\sup_{t\leq -2}|{h}(t)|+\sup_{t\leq -2}\frac{|t|}{\log|t|}|{h}'(t)|<1.
\ea

Set
\ba\label{z}
z(t,r(y)):=\sum_{j=1}^k(-1)^{j+1}w(r(y)-\xr_j(t))-1,
\ea
we consider the following projected version of equation \eqref{mainequation} in terms of $\psi$
\ba\label{mainpro}\left\{\BAL
\psi_t&=\frac{(N-1)^2}{(y^2+1)^{\frac{1}{2(N-1)}}}\frac{d}{dy}\left((y^2+1)^{\frac{2N-3}{2(N-1)}}\psi_y\right)+f'(z)\psi\\
&\qquad\qquad+E+N(\psi)-\sum_{i=1}^kd_i(t)w'(r(y)-\xr_i(t))\qquad\qquad\text{in}\;\;(-\infty,-T)\times(0,\infty)\\
\psi(t,0)&=\psi_y(t,0)=0\phantom{+E+N(\psi)-\sum_{i=1}^kd_i(t)w'(r(y)-\xr_i(t))}\,\qquad\text{in}\;\;\;(-\infty,-T)\\
\EAL\right.
\ea
and

\ba\label{orthcon}
\int_0^\infty\psi(t,y)w'(r(y)-\xr_j(t))r(y)dy=0\quad\forall\;\;j=1,...,k\quad\text{and}\quad t\in (-\infty,-T),
\ea
where
\ba\label{error}\BAL
E&=\sum_{i=1}^k(-1)^{i+1}\left(w'(r(y)-\xr_i(t))\xr_i'(t)+\frac{N-1}{r(y)}w'(r(y)-\xr_i(t))\right)+f(z(t,r(y)))\\
&\qquad-\sum_{i=1}^k(-1)^{i+1}(f(w(r(y)-\xr_i(t)))+\frac{w''(r(y)-\xr_i(t))}{1+y^2}),
\EAL
\ea

\bal
N(\psi)=f(z(t,r(y))+\psi(t,y))-f(z(t,r(y)))-f'(z)\psi.
\eal
The functions $d_i(t)$ are chosen such that $\psi$ satisfies the orthogonality condition \eqref{orthcon}, namely

\ba\label{ole}\BAL
&\sum_{i=1}^kd_i(t)\int_0^\infty w'(r(y)-\xr_i(t))w'(r(y)-\xr_j(t))r(y)dy=-\xr_j'(t)\int_0^\infty \psi w''(r(y)-\xr_j(t))r(y)dy\\
&\qquad\qquad(N-1)\int_0^\infty\psi(t,y)(w''(r(y)-\xr_j(t))y)_y dy\\
&+\qquad\qquad\int_0^\infty f'(z(t,r(y)))\psi(t,y)w'(r(y)-\xr_j(t))r(y)dy\\
&\qquad\qquad+\int_0^\infty(E+N(\psi))w'(r(y)-\xr_j(t))r(y)dy \qquad\qquad \forall\; j=1,...,k\quad\text{and} \;\;t<-T.
\EAL
\ea
Later we will choose $\xr_i(t)$ such that $d_i(t) = 0,$ for any $i = 1, ..., k.$

In the following lemma, we find an estimate for the error term $E=E(t,r(y))$ in \eqref{error}.
\begin{lemma}
Let $T_0>2,$ $0<\xs<\sqrt{2},$ we define
\ba\xF(t,r):=\left\{\BAL
&e^{\xs(-r+\rho_{j-1}^0(t))}+e^{\xs(r-\rho_{j+1}^0(t))}&&\mathrm{if}\;\;\frac{\rho_j^0(t)+\rho_{j-1}^0(t)}{2}\leq r\leq \frac{\rho_j^0(t)+\rho_{j+1}^0(t)}{2},\;j=2,...,k,\\
&e^{\xs(r-\rho_{2}^0(t))}&&\mathrm{if}\;\; 1< r\leq \frac{\rho_1^0(t)+\rho_{2}^0(t)}{2},
\EAL\right.\label{bound}
\ea
with $\rho_{k+1}^0=\infty.$ Then there exists a uniform constant $C>0$ which depends only on $k,$ such that
\bal|E(t,r)|\leq C\left(\frac{\log |t|}{|t|}\right)^{\frac{\sqrt{2}-\xs}{2\sqrt{2}}}\xF(t,r)\quad\forall (t,r)\in(-\infty, -T_0]\times(1,\infty),\eal
where $E$ is the error term in \eqref{error}.\label{remark}
\end{lemma}
\begin{proof}

First we note that
\bal
|\rho_j'(t)+\frac{N-1}{r}|\leq C\frac{\log |t|}{|t|}\quad\text{if}\;\; \rho_1^0-\frac{\sqrt{2}+\xs}{\sqrt{2}-\xs}\eta\leq r\leq \rho_k^0+\frac{\sqrt{2}+\xs}{\sqrt{2}-\xs}\eta,
\eal

\bal
\frac{w'(r-\rho_j^0(t))}{\xF}\leq C\left(\frac{\log |t|}{|t|}\right)\quad\mathrm{if}\;\; r\geq \rho_k^0+\frac{\sqrt{2}+\xs}{\sqrt{2}-\xs}\eta\;\;\text{or}\;\;r\leq \rho_1^0-\frac{\sqrt{2}+\xs}{\sqrt{2}-\xs}\eta,
\eal

and

\ba\label{mainbound}
\frac{w'(r-\rho_j(t))}{\xF}\leq C\left(\frac{\log{|t|}}{|t|}\right)^{-\frac{\xs}{\sqrt{2}}}\quad\forall r>1,
\ea
where $C$ depends only on $\xs,k$ and the constant $c_0$ in \eqref{lemodesimpl1}.  Combining all above, we obtain

\bal
\bigg|\rho_j'(t)+\frac{N-1}{r(y)}\bigg|w'(r-\rho_j(t))+\bigg|\frac{w''(r(y)-\xr_i(t))}{1+y^2}\bigg|\leq C\xF.
\eal
where $C$ depends only on $\xs,k,c_0$ and $N.$

Next assume that
\bal\BAL
\frac{\rho_j^0(t)+\rho_{j-1}^0(t)}{2}\leq r&\leq \frac{\rho_j^0(t)+\rho_{j+1}^0(t)}{2}&&\;\text{if}\;\;j=2,...,k,\\
\text{or}\quad r&\leq \frac{\rho_1^0(t)+\rho_{2}^0(t)}{2}&&\;\text{if}\;\;j=1.
\EAL
\eal
If $i\leq j-1,$ by our assumptions on $\rho_i,$ there exists a uniform constant $C>0$ such that

\bal
|w(r-\rho_i(t))-1|\leq Ce^{\sqrt{2}(-r+\rho_{j-1}^0(t))}.
\eal
Similarly if $i\geq j+1$

\bal
|w(r-\rho_i(t))+1|\leq Ce^{\sqrt{2}(r-\rho_{j+1}^0(t))}.
\eal
We set
\bal
g=\left\{\BAL
&\sum_{i=1}^{j-1}(-1)^{i+1}\left(w(r-\rho_i)-1\right)+\sum_{i=j+1}^k(-1)^{i+1}\left(w(r-\rho_i)+1\right)&&\quad\text{if}\;\;j=2,...,k-1,\\
&\sum_{i=2}^k(-1)^{i+1}\left(w(r-\rho_i)+1\right)&&\quad\text{if}\;\;j=1,\\
&\sum_{i=1}^{k-1}(-1)^{i+1}\left(w(r-\rho_i)+1\right)&&\quad\text{if}\;\;j=k.
\EAL\right.
\eal
Then
\ba\label{213}\BAL
&\left|f\left(g+(-1)^{j+1}w(r-\rho_j(t))\right)-\sum_{i=1}^{k}(-1)^{i+1}f(w(r-\rho_i(t))\right|\\ 
&\leq C\left(\sum_{i=1}^{j-1}|w(r-\rho_i)-1|+\sum_{i=j+1}^k(-1)^{i+1}|w(r-\rho_i)+1|\right).
\EAL
\ea
Combining all above and using the properties of $\rho,$ we can easily reach the desired result.
\end{proof}

\section{The linear problem}
\subsection{Existence} Let $T_0>2,$ $s<t<-T_0$ and $\mathcal{C}_\xF((s,t)\times(0,\infty)$ be the space of continuous functions
with norm
\be
||u||_{\mathcal{C}_\xF((s,t)\times(0,\infty))}=\left|\left|\frac{u}{\xF}\right|\right|_{L^\infty((s,t)\times(0,\infty))},\label{norm11111}
\ee
where $\xF$ has been defined in \eqref{bound}.

\begin{theorem}
Let $s<-T_0$ and $g\in\mathcal{C}_\xF((s,t)\times(0,\infty).$ Then there exists a unique weak solution $u\in\mathcal{C}_\xF((s,-T_0)\times(0,\infty))$ of the problem

\ba\label{mainpro*}\left\{\BAL
v_t&=\frac{(N-1)^2}{(y^2+1)^{\frac{1}{2(N-1)}}}\frac{d}{dy}\left((y^2+1)^{\frac{2N-3}{2(N-1)}}v_y\right)+f'(z(t,r(y)))v+g
\;\;\;\;\;\text{in}\;(s,-T_0)\times(0,\infty)\\
v(t,0)&=v_y(t,0)=0\quad\text{in}\;\;\;(s,-T_0).
\EAL\right.
\ea
Furthermore $v,v_y\in C^{\xa}((s,t))\times[0,\infty)$ for some $\xa\in(0,1)$ and there exists $C>0$ depending only on $N$ and $t-s$ such that

\ba
||v||_{\mathcal{C}_\xF((s,t)\times(0,\infty))}\leq C||g||_{\mathcal{C}_\xF((s,t)\times(0,\infty))}.\label{normest}
\ea
\end{theorem}
\begin{proof}
Let $l>0,$ $\overline{t}\in (s,-T_0],$ $t\in(s,\overline{t}),$ $\rho_{k+1}^0=\infty$ and $\rho_{0}^0=-\infty.$ Set
\bal
\xf_j(t,r(y))=2 C(g)e^{l (t-s)}\left(e^{\xs\left(r(y)-\rho_{j+1}^0(t)\right)}+e^{\xs\left(-r(y)+\rho_{j-1}^0(t)\right)}\right),
\eal
where
\bal
C(g)=\left|\left|g\right|\right|_{\mathcal{C}_\xF((s,\overline{t})\times(0,\infty))},
\eal
We can easily check that $\xf_j\geq C(g)\xF.$ In addition, there holds

\ba\label{calculation}\BAL
\frac{\partial\xf_j}{\partial t}&-\frac{(N-1)^2}{(y^2+1)^{\frac{1}{2(N-1)}}}\frac{\partial}{\partial y}\left((y^2+1)^{\frac{2N-3}{2(N-1)}}\frac{\partial\xf_j}{\partial y}\right)-f'(z(t,r(y)))\xf_j\\
&=\frac{\partial\xf_j}{\partial t}-\frac{y^2}{1+y^2}\frac{\partial^2\xf_j}{\partial r^2}-\frac{N-1}{r(y)} \frac{\partial\xf_j}{\partial r}-f'(z(t,r(y)))\xf_j\\
&\geq\xf_j\bigg(l+\xs(-\rho_{j+1}^{0'}+\rho_{j-1}^{0'})-\frac{\xs^2y^2}{1+y^2}-\frac{\xs(N-1)}{r(y)}-f'(z(t,r(y)))\bigg)\geq \xf_j,
\EAL
\ea
provided $l>0$ is large enough. Here $l$ depends only on $\xs,N$ and $k.$  Hence, using $\xf_j$ like a barrier and standard parabolic theory, we may prove the existence of a solution of \eqref{mainpro*} such that $v,v_y\in C^{\xa}((s,-T_0))\times[0,\infty)$ for some $\xa\in(0,1)$ as well as
\bal
|v(t,y)|\leq  \xf_j(t,r(y)) \quad\text{for}\;j=1,...,k.
\eal
The desired result follows easily by the above inequality.
\end{proof}

\subsection{A priori estimates for the solution of the problem (\ref{mainpro*})}

 We will establish in this subsection a priori estimates for the
solutions  $\psi^s$ of (\ref{mainpro*}) that are independent on $s.$
\begin{lemma}

Let  $g\in \mathcal{C}_\xF\left((s,-T)\times(0,\infty)\right)$ and $\psi^s\in \mathcal{C}_\xF\left((s,-T_0)\times(0,\infty)\right)$ be a solution of the problem (\ref{mainpro*}) which satisfies the orthogonality conditions

\be
\int_0^\infty \psi^s(t,y)w'(r(y)-\rho_i(t))r(y)dy=0,\qquad\forall i=1,...,k,\;s<t<-T_0. \label{cond}
\ee

Then there exists a uniform constant $T_1>0$ such that for any $t\in (s,-T_1],$ the following estimate is valid
\be
||\psi^s||_{\mathcal{C}_\xF\left((s,t)\times(0,\infty)\right)}\leq C\left(||g||_{\mathcal{C}_\xF\left((s,t)\times(0,\infty)\right)}\right),\label{estfix}
\ee
where $C>0$ is a uniform constant.\label{mainlem}
\end{lemma}
\begin{proof}
Set
\bal
A_{j,R}^{(s,t)}=\left\{(\tau,y)\in(s,t)\times(0,\infty):\;|r(y)-\rho_j^0(\tau)|<R+1\right\}
\eal
and
\bal
A_j^{(s,t)}=\left\{(\tau,y)\in(s,t)\times(0,\infty):\;\max\left(\frac{\rho_j^0(\tau)+\rho_{j-1}^0(\tau)}{2},1\right)< r(y)< \frac{\rho_j^0(\tau)+\rho_{j+1}^0(\tau)}{2}\right\},
\eal
with $\xr_0^0=-\infty$ and $\rho_{k+1}^0=\infty.$

We will prove \eqref{estfix} by contradiction. Let $\{s_i\},\;\{\overline{t}_i\}$ be sequences such that $s_i< \overline{t}_i\leq -T_0.$  We also assume that there exists $g_{i}$ such that $\psi_i$  solve (\ref{mainpro*}) with $s=s_i,$  $g=g_i$ and satisfies \eqref{cond}. Finally we assume that
\ba\label{contr}
\left|\left|\psi_i\right|\right|_{\mathcal{C}_\xF((s_i,\overline{t}_i)\times(0,\infty))}=1\quad\text{and}\quad
\left|\left|g_i\right|\right|_{\mathcal{C}_\xF((s_i,\overline{t}_i)\times(0,\infty))}
\rightarrow0.
\ea

First we note that, by \eqref{normest}, we may assume that

\ba\label{assumption1}
s_i+1<\overline{t}_i,\quad s_i\downarrow-
\infty\quad\text{and}\quad\overline{t}_i\downarrow-\infty
\ea
To reach a contradiction, we need the following assertion.

\medskip

\noindent\textbf{Assertion 1.} \emph{Let $R>0$ then we have}
\be
\lim_{i\rightarrow\infty}\left|\left|\frac{\psi_i}{\xF}\right|\right|_{L^\infty(A_{j,R}^{(s_i,\overline{t}_i)})}=0\;\;\;\;\;\forall j=1,...,k.\label{contr2}
\ee

\medskip

Indeed, let us first assume that \eqref{contr2} is valid. Set
\bal
\xm_{i,j}:=\left|\left|\frac{\psi_i}{\xF}\right|\right|_{L^\infty(A_{j,R}^{(s_i,\overline{t}_i)})}\longrightarrow_{i\rightarrow\infty}0\;\;\;\;\;\forall j=1,...k.
\eal
Let $j\in\{1,...,k\}$ and

\bal
\frac{\rho_j^0(t)+\rho_{j-1}^0(t)}{2}&\leq r(y)\leq \frac{\rho_j^0(t)+\rho_{j+1}^0(t)}{2}.
\eal
By our assumptions on $\rho_i,$ we have 

\bal
|w(r-\rho_n(t))-1|\leq 2e^{\sqrt{2}(-r+\rho_{n}(t))}&\leq e^{-\frac{\sqrt{2}}{2}(\rho_{j}-\rho_{j-1}(t))} \leq C\left(\frac{\log|t|}{|t|}\right)^{\frac{1}{2}}\quad\text{if}\;\;1\leq n\leq j-1,\\
|w(r-\rho_n(t))+1|\leq 2e^{\sqrt{2}(r-\rho_{n}(t))}&\leq C\left(\frac{\log|t|}{|t|}\right)^{\frac{1}{2}}\quad\text{if}\;\; j+1\leq n\leq k
\eal
and
\bal
|w(r-\rho_j(t))|\geq w(R+1) \quad\text{if}\;\;|r-\rho_j(t)|>R+1.
\eal
Combining the last three displays, we may show that for any $0<\xe<2$ there exists $i_0\in \mathbb{N}$ and $R>0$ such that

\be
-f'(z(r,x))\geq 2-\xe,\;\;\forall t\leq \overline{t}_i,\;\;x\in \mathbb{R}\setminus\cup_{j=1}^kA_{j,R}^{(s_i,t_i)}\;\;\; \text{and}\;\;\;i\geq i_0.\label{fr11}
\ee

Next, consider the function
\bal
\overline{\xf}_{i,j}(t,r)&=2\left(e^{\xs\left(r-\rho_{j+1}^0(t)\right)}+e^{\xs\left(-r+\rho_{j-1}^0(t)\right)}+e^{\xs\left(4\tilde M-r+\rho_{2}^0(t)\right)}\right)\\
&\times\left(\left|\left|g_{i}\right|\right|_{\mathcal{C}_\xF((s_i,\overline{t}_i)\times(0,\infty))}+
\sup_{1\leq j\leq k}\xm_{i,j}\right)
\eal
and note that

\bal
\psi_i(t,y)\leq \xF(t,r(y))\sup_{1\leq j\leq k}\xm_{i,j}\leq \overline{\xf}_{i,j}(t,r(y)),\quad\forall\; (t,y)\in \overline{\cup_{j=1}^kA_{j,R}^{(s_i,\overline{t}_i)}}.
\eal

Now, let $\xe>0,\;\tilde{M}>1$ be such that $2-\xe^2>\xs^2+\frac{N-1}{\tilde{M}}$ and $M_0=\min(2-\xe^2-\xs^2+\frac{N-1}{\tilde{M}},1).$ In view of \eqref{calculation} and \eqref{fr11}, we may deduce that there exists $i_0\in\BBN$ such that

\bal
\frac{\partial\overline{\xf}_{i,j}}{\partial t}&-\frac{(N-1)^2}{(y^2+1)^{\frac{1}{2(N-1)}}}\frac{\partial}{\partial y}\left((y^2+1)^{\frac{2N-3}{2(N-1)}}\frac{\partial\overline{\xf}_{i,j}}{\partial y}\right)-f'(z(t,r(y)))\overline{\xf}_{i,j}\geq \frac{M_0}{2}\overline{\xf}_{i,j}\geq\frac{M_0}{2} g_i\\
&\quad\forall\; (t,x)\in ((s_i,\overline{t}_i)\times(0,\infty))\setminus\cup_{j=1}^kA_{j,R}^{(s_i,\overline{t}_i)},\;\;j=1,...k\;\;\;\text{and}\;\;i\geq i_0.
\eal
Hence, we may use $\overline{\xf}_{i,j}$ like a barrier to obtain

\bal
|\psi_i(t,y)|\leq \frac{2}{M_0}|\overline{\xf}_{i,j}(t,r(y))|,\;\;\forall(t,y)\in ((s_i,\overline{t}_i)\times(0,\infty))\setminus\cup_{j=1}^kA_{j,R}^{(s_i,\overline{t}_i)},\;\;j=1,...k,\;\;\;i\geq i_0.
\eal
This implies 
\bal
1=\left|\left|\psi_i\right|\right|_{\mathcal{C}_\xF((s_i,\overline{t}_i)\times(0,\infty))}\leq
\frac{4}{M_0}\left(\left|\left|g_i\right|\right|_{\mathcal{C}_\xF((s_i,\overline{t}_i)\times(0,\infty))}
+\sup_{1\leq j\leq k}\xm_{i,j}\right),
\eal
which is clearly a contradiction if we choose $i$ large enough.

\medskip

\noindent\textbf{Proof of Assertion 1.} We will prove Assertion 1 by contradiction in four steps.

\medskip

Let us give first the contradict argument and some notations.  Set
\bal
\tilde{\psi}_i(t,r)=\psi_i(t,\sqrt{r^{2(N-1)}-1})\quad \text{and}\quad \tilde g_i(t,r)=g_i(t,\sqrt{r^{2(N-1)}-1}).
\eal
Then, by \eqref{chvar}, we have
\bal
\frac{\partial\tilde{\psi}_i(t,r)}{\partial t}-\frac{r^{2(N-1)}-1}{r^{2(N-1)}}\frac{\partial^2\tilde{\psi}_i(t,r)}{\partial r^2}-\frac{N-1}{r} \frac{\partial\tilde{\psi}_i(t,r)}{\partial r}&-f'(z(t,r))\tilde{\psi}_i(t,r)\\
&=\tilde g_i(t,r)\quad\text{in}\;(s_i,\overline{t}_i)\times(1,\infty)
\eal

We assume that (\ref{contr2}) is not valid. Then there exists $j\in\{1,...,k\}$ and $\xd>0$ such that
\bal
\left|\left|\frac{\tilde\psi_i}{\xF}\right|\right|_{L^\infty(A_{j,R}^{(s_i,\overline{t}_i)})}>\xd>0,\;\;\forall i\in\mathbb{N}.
\eal
Let $(t_i,r_i)\in A_{j,R}^{(s_i,\overline{t}_i)}$ such that
\be
\left|\frac{\tilde\psi_i(t_i,r_i)}{\xF(t_i,r_i)}\right|>\xd.\label{667}
\ee
In addition, by definition of $\xF,$ there holds
\be
\xF(t_i,r_i)=e^{\xs(-r_i+\rho_{j-1}(t_i))}+e^{\xs(r_i-\rho_{j+1}(t_i))}.\label{obs1}
\ee

Next, set $r_i=x_i+\rho_j(t_i)$ and
\bal
\xf_i(t,x)=\frac{\tilde\psi_i(t+t_i,x+x_i+\rho_j(t+t_i))}{\xF(t_i,x_i+\rho_j(t_i))}.
\eal
Then $\xf_i$ satisfies
\ba\label{eq11}\left\{\BAL
&(\xf_i)_t=\; \frac{(x+x_i+\rho_j(t+t_i))^{2(N-1)}-1}{(x+x_i+\rho_j(t+t_i))^{2(N-1)}}(\xf_i)_{xx}+\frac{N-1}{x+x_i+\rho_j(t+t_i)}(\xf_i)_x\\
&+\rho_j'(t+t_i)(\xf_i)_x+f'(z(t+t_i,x+x_i+\rho_j(t+t_i)))\xf_i\\ &+\frac{g_i(t+t_i,x+x_i+\rho_j(t+t_i))}{\xF(t_i,x_i+\rho_j(t_i))} && \mathrm{in}\;\;\;\xG^{(s_i,t_i)}_j\\
&\xf_i(s_i-t_i,x-x_i-\rho_j(t+t_i))=0,&&\mathrm{in}\;(1,\infty),
\EAL\right.
\ea
where
\bal
\xG^{(s_i,t_i)}_j=\{(t,x)\in(s_i-t_i,0]\times\mathbb{R}:\;1-x_i-\rho_j(t+t_i)<x\}.
\eal
Also set
\ba
\nonumber
&B_{t_i,n,j}=\Bigg\{(t,x)\in(s_i-t_i,0]\times\mathbb{R}:
\;\max\left(\frac{\rho_n^0(t+t_i)+\rho_{n-1}^0(t+t_i)}{2},1\right)-\rho_j(t+t_i)-x_i\\ \nonumber
&\leq x\leq
 \frac{\rho_n^0(t+t_i)+\rho_{n+1}^0(t+t_i)}{2}-\rho_j(t+t_i)-x_i\Bigg\}\quad\text{for}\;\;n=1,...,k,
\ea
and
\bal
B_{t_i,n,j}^M=B_{t_i,n,j}\cap\left\{(t,x)\in(s_i-t_i,0]\times\mathbb{R}: |x+\rho_j(t+t_i)+x_i-\rho_n^0(t+t_i)|>M\right\},
\eal
where $n=1,....,k$ and $M>0.$
We note here that $|x_i|<R+1$ for any $i\in\mathbb{N}$ and  $|\xf_i(0,0)|=\left|\psi_i(t_i,y_i)/\xF(t_i,y_i)\right|>\xd>0.$
Finally, by \eqref{normest} and \eqref{contr}, we have that
\bal
\liminf t_i-s_i>\infty.
\eal

In the sequel, assume that $x_i\rightarrow x_0\in B_{R+1}(0)$ and $\lim_{i\rightarrow\infty} t_i-s_i=\infty$ (otherwise take a subsequence).

\medskip

\noindent\textbf{Step 1} We assert that
$\xf_i\rightarrow\xf$ locally uniformly, $\xf(0,0)>\xd$  and $\xf$ satisfies
\be
\xf_t=\; \xf_{xx} +f'(w(x+x_0))\xf,\qquad \mathrm{in}\;\;\;(-\infty,0]\times\mathbb{R}.\label{equ1}
\ee

\medskip

Let $(t,x)\in B_{t_i,n,j}$ and $1\leq n\leq k.$ By \eqref{forma4}, \eqref{contr} and \eqref{obs1} we have that
\ba\BAL
&|\xf_i(t,x)|\leq \left|\frac{\psi_i(t+t_i,x+x_i+\rho_j(t+t_i))}{\xF(t_i,x_i+\rho_j(t_i))}\right|\leq \left|\frac{\xF(t+t_i,x+x_i+\rho_j(t+t_i))}{\xF(t_i,x_i+\rho_j(t_i))}\right|\\
&\;\;\leq C_0(\xb,||h||_{L^\infty},\sup_{1\leq j\leq k}|\xg_j|,\xs,R)\left(\frac{|t_i|\log|t+t_i|}{|t+t_i|\log|t_i|}\right)^{\frac{\xs}{\sqrt{2}}}e^{\xs|x+\rho_j(t+t_i)-\rho_n(t+t_i)|},
\EAL\label{constant}
\ea
for any $i\in\mathbb{N}$ and $(t,x)\in B_{t_i,n,j}.$ In addition, there holds
\bal
\cup_{i=1}^\infty B_{t_i,j,j}=(-\infty,0]\times\mathbb{R}.
\eal
The desired result follows by the last two displays and standard regularity results for parabolic equations.

\medskip
\noindent\textbf{Step 2} In this step, we prove the following orthogonality condition 

\be
\int_{\mathbb{R}}\xf(t,x)w'(x+x_0)dx=0\;\;\;\forall t\in (-\infty,0].\label{oth}
\ee

By assumption, we have that
\be
\int_0^\infty \psi_i(t+t_i,y)w'(r(y)-\rho_\ell(t+t_i))r(y)dy=0\qquad\forall \ell=1,...,k,\;s_i-t_i<t\leq0. \label{cond2}
\ee
By change of variables with  $r=(y^{2}+1)^{\frac{1}{2(N-1)}}$, we obtain
\bal
0&=\int_0^\infty \psi_i(t+t_i,y)w'(r(y)-\rho_\ell(t+t_i))r(y)dy\\
&=(N-1)\int_1^\infty \tilde\psi_i(t+t_i,r)w'(r-\rho_\ell(t+t_i))\frac{r^{2(N-1)}}{\sqrt{r^{2(N-1)}-1}}dr.
\eal
This implies 
\bal\int_{1-x_i-\xr_j(t+t_i)}^\infty \xf_i(t,x)w'(x+x_i+\rho_j(t+t_i)-\rho_\ell(t+t_i))\frac{(x+x_i+\rho_j(t+t_i))^{2(N-1)}}{\sqrt{(x+x_i+\rho_j(t+t_i))^{2(N-1)}-1}}dx=0,
\eal
for any $t\in (s_i-t_i,0]$ and $\ell=1,...,k.$

Let $t\in \cap_{i=i_0}^\infty(s_i-t_i,0],$ for some $i_0\in\mathbb{N}$ and

\bal
\nonumber
B_{t,t_i,n,j}&=\Bigg\{x\in\mathbb{R}:
\;\max\left(\frac{\rho_n^0(t+t_i)+\rho_{n-1}^0(t+t_i)}{2},1\right)-\rho_j(t+t_i)-x_i\\ \nonumber
&\qquad\qquad\leq x\leq
 \frac{\rho_n^0(t+t_i)+\rho_{n+1}^0(t+t_i)}{2}-\rho_j(t+t_i)-x_i\Bigg\} \quad \text{for}\;\;n=1,...,k.
\eal

By \eqref{constant}, there exists a positive constant, which is independent of $x\in\BBR$ and $i,$ such that for any $x\in \BBR$ there hold
\ba\label{est2}\BAL
\chi_{B_{t,t_i,j,j}}(x)\left|\xf_i(t,x)w'(x+x_i)\frac{(x+x_i+\rho_j(t+t_i))^{N-1}}{\sqrt{(x+x_i+\rho_j(t+t_i))^{2(N-1)}-1}}\right|\leq Ce^{-(\sqrt{2}-\xs)|x|}\quad\forall j=2,...,k 
\EAL
\ea
and

\ba\label{est3}\BAL
&\chi_{B_{t,t_i,1,1}\cap\{x>2-x_i-\xr_1(t+t_i)\}}(x)\left|\xf_i(t,x)w'(x+x_i)\frac{(x+x_i+\rho_1(t+t_i))^{N-1}}{\sqrt{(x+x_i+\rho_1(t+t_i))^{2(N-1)}-1}}\right|\\
&\qquad\leq Ce^{-(\sqrt{2}-\xs)|x|}.
\EAL
\ea
In addition, by (\ref{constant}), we can easily prove that

\ba\label{est2a}
\lim_{i\to\infty}\int_{1-x_i-\xr_1(t+t_i)}^{2-x_i-\xr_1(t+t_i)} |x|^\xa \left|\xf_i(t,x)w'(x+x_i)\frac{(x+x_i+\rho_1(t+t_i))^{N-1}}{\sqrt{(x+x_i+\rho_1(t+t_i))^{2(N-1)}-1}}\right|dx=0
\ea
for any $\xa\in\{1,...,N-1\}.$

Let  $n>j.$ By \eqref{constant}, the assumptions on $\rho$ and the fact that $|x_i|<R+1,$ we have

\ba\BAL
&\int_{B_{t,t_i,n,j}}|x|^\xa\left|\xf_i(t,x)w'(x+x_i)\frac{(x+x_i+\rho_j(t+t_i))^{N-1}}{\sqrt{(x+x_i+\rho_j(t+t_i))^{2(N-1)}-1}}\right|dx\\
&\leq C_0 \int_{{\frac{\rho_n^0(t+t_i)+\rho_{n-1}^0(t+t_i)}{2}-\rho_j(t+t_i)-x_i}}^{\frac{\rho_n^0(t+t_i)+
\rho_{n+1}^0(t+t_i)}{2}-\rho_j(t+t_i)-x_i}|x|^\xa e^{-\sqrt{2}x+\xs|x+\rho_j(t+t_i)-\rho_n(t+t_i)|}dx\\
&\leq C (\log|t+t_i|)^{\xa}\left(\frac{\log|t+t_i|}{|t+t_i|}\right)^{\frac{(\sqrt{2}-\xs)}{2\sqrt{2}}}\rightarrow_{i\rightarrow\infty}0.
\EAL
\label{est1}
\ea
Similarly the estimate \eqref{est1} is valid for $n=1,...,j-1.$

By (\ref{est2})-(\ref{est1}), we have that
\bal
0&=\lim_{i\rightarrow\infty}\frac{1}{(\rho_j(t+t_i))^{N-1}}
\int_{1-x_i-\rho_j(t+t_i)}^\infty\xf_i(t,x)w'(x+x_i)\frac{(x+x_i+\rho_j(t+t_i))^{2(N-1)}}{\sqrt{(x+x_i+\rho_j(t+t_i))^{2(N-1)}-1}}dx\\
&=\int_{\mathbb{R}}\xf(t,x)w'(x+x_0)dx
\eal
and the proof of this assertion follows.

\medskip
\noindent\textbf{Step 3} In this step we prove the following assertion:

\medskip

 \emph{Let $0<\tilde\xs<\sqrt{2}.$ Then there exists $C=C(R,\tilde\xs,\xs,k)>0,$ such that}

\be
|\xf(t,x)|\leq Ce^{-\tilde \xs|x|}\quad\forall(t,x)\in(-\infty,0]\times \mathbb{R}.\label{bound1}
\ee

\medskip

By \eqref{constant}, we have that

\be
|\xf(t,x)|\leq Ce^{\xs|x|},\;\;\forall(t,x)\in(-\infty,0]\times \mathbb{R}.\label{bound1b}
\ee

Let $\xe>0$ and let $M>$ be large enough such that $-f'(w(x+x_0))\geq \max(\tilde\xs, \frac{\xs+\sqrt{2}}{2})^2$ for any $|x|>M.$ We set $G(x)=\tilde M e^{-\tilde\xs |x|}+\xe e^{\frac{\xs+\sqrt{2}}{2}|x|}.$ By \eqref{bound1}, there exists $\tilde M$ large enough such that

\bal
\max(|\xf(t,x)|-\tilde M e^{-\tilde\xs x},0)=0\quad\forall (t,x)\in(-\infty,0]\times\{|x|=M\}.
\eal
Combining all above, we may invoke the maximum principle to infer 

\bal
|\xf(t,x)|\leq \tilde M e^{-\tilde\xs x}+\xe e^{\frac{\xs+\sqrt{2}}{2}|x|}\quad \forall (t,x)\in(-\infty,0]\times\{|x|\geq M\}\;\text{and}\;\xe>0.
\eal
Letting $\xe\to0^+,$ we obtain the desired result.

\medskip
\noindent\textbf{ Step 4}
In this step we prove assertion \eqref{contr2}.

\medskip

Consider the Hilbert space \bal H=\{\xz\in H^1(\mathbb{R}):\;\int_{\mathbb{R}}\xz(x)w'(x)dx=0.\}\eal
Then it is well known that there exists $c>0$ such that the following inequality is valid
\be
\int_{\mathbb{R}}|\xz'(x)|^2-f'(w)|\xz|^2\geq c\int_{\mathbb{R}}|\xz(x)|^2dx \quad\forall\xf(x)\in H\cap L^2(\mathbb{R}).\label{Heq}
\ee
Thus if we multiply (\ref{equ1}) by $\xf$ and integrate with respect $x$ we have
\bal
0&=\frac{1}{2}\int_{\mathbb{R}}(\xf^2)_tdx+\int_{\mathbb{R}}|\xf_x|^2dx-f'(w(x))|\xf|^2dx\\
&\geq\frac{1}{2}\int_{\mathbb{R}}(\xf^2)_tdx+c\int_{\mathbb{R}}|\xf(t,x)|^2dx.
\eal

Set $a(t)=\int_{\mathbb{R}}|\xf(t,x)|^2dx,$ then
\bal
a'(t)\leq-2ca(t)\Rightarrow a(t)\geq a(0)e^{2c|t|},
\eal
which is a contradiction, since $a(0)>0$ by Step 1 and 
\bal
||e^{\xs|x|}\xf||_{L^\infty((-\infty,0]\times\mathbb{R})}<C.
\eal

\end{proof}

\subsection{Problem \eqref{mainpro*} with $g(t,r)=h(t,r)- \sum_{j=1}^kd_j(t)w'(r-\rho_j(t))$} Let $T\geq 2.$
In this subsection, we study the following problem.

\ba\left\{\BAL
\psi_t&=\frac{(N-1)^2}{(y^2+1)^{\frac{1}{2(N-1)}}}\frac{d}{dy}\left((y^2+1)^{\frac{2N-3}{2(N-1)}}\psi_y\right)+f'(z(t,r(y)))\psi &&\\
&\qquad\qquad+h(t,y)-\sum_{j=1}^k d_i(t)w'(r-\rho_j(t)),\quad &&\mathrm{in}\;\;\;(s,-T]\times(0,\infty),\\
\psi(s,0)&=\psi_y(s,0)=0\quad&&\text{in}\;\;\;(s,-T]
\EAL\right.\label{proci}
\ea
where $h\in\mathcal{C}_\xF((s,-T)\times(0,\infty))$ and $d_i(t)$ satisfies the following (nearly diagonal) system
\ba\BAL
&\sum_{i=1}^kd_i(t)\int_0^\infty w'(r(y)-\xr_i(t))w'(r(y)-\xr_j(t))r(y)dy=-\xr_j'(t)\int_0^\infty \psi w''(r(y)-\xr_j(t))r(y)dy\\
&\qquad\qquad+(N-1)\int_0^\infty\psi(t,y)\left(w''(r(y)-\xr_j(t))y\right)_y dy\\
&\qquad\qquad+\int_0^\infty f'(z(t,r(y)))\psi(t,y)w'(r(y)-\xr_j(t))r(y)dy\\
&\qquad\qquad+\int_0^\infty h(t,r)w'(r(y)-\xr_j(t))r(y)dy\qquad\qquad\quad\forall\; j=1,...,k\quad\text{and} \;\;s<t<-T.
\EAL
\label{ci(t)}
\ea
We note here that if $\psi$ is a solution of \eqref{proci} and $d_i(t)$ satisfies the above system then $\psi$ satisfies the orthogonality conditions
\ba\label{orthcond}
\int_{0}^\infty\psi(t,r)w'(r(y)-\rho_i(t))r(y)dy=0\quad\forall i=1,...,k\;\text{and}\;s<t<-T.
\ea

\medskip
The main result of this subsection is the following
\begin{lemma}
Let $h\in\mathcal{C}_\xF((s,-T)\times[0,\infty))$. There exist a uniform constant $T_0\geq T\geq2,$ and a unique solution  $\psi^s$ of problem \eqref{proci}. Furthermore, we have that $\psi^s$ satisfies the orthogonality conditions and the following estimate
 \ba
||\psi^s||_{\mathcal{C}_\xF((s,t)\times(0,\infty))}\leq C||h||_{\mathcal{C}_\xF((s,t)\times(0,\infty))}\quad\forall t\in (s,-T_0),\label{estfixmef}
\ea
where $C>0$ is a uniform constant.\label{cilemma*}
\end{lemma}
To prove the above Lemma we need the following result.

\begin{lemma}
Let $T>0$ be large enough, $h\in\mathcal{C}_\xF((s,-T)\times\mathbb{R})$ and $\psi\in \mathcal{C}_\xF((s,-T)\times[0,\infty)).$ Then there exists $d(t)=(d_1(t),...,d_k(t))$ such that the nearly diagonal system \eqref{ci(t)} holds. Furthermore, there exists a positive constant $c$ independent of $T,\;s,\;t,\;\psi,\;h$ such that
\ba\label{ciest}\BAL
|d_i(t)|&\leq c\left(\frac{\log|t|}{|t|}\right)^{\frac{\sqrt{2}+\xs}{2\sqrt{2}}}\left|\left|\psi\right|\right|_{\mathcal{C}_\xF((s,-T)\times(0,\infty))}\\
&\qquad+
c\left(\frac{\log|t|}{|t|}\right)^{\frac{\xs}{\sqrt{2}}}||h||_{\mathcal{C}_\xF((s,-T)\times(0,\infty))}\;\;\; \forall t\in [s,-T]\;\text{and}\; i=1,...,k
\EAL
\ea
and

\bal
\left|\frac{d_i(t)w'(r-\rho_i(t))}{\xF(t,r)}\right|&\leq c\left(\frac{\log|t|}{|t|}\right)^{\frac{1}{2}}\left|\left|\psi\right|\right|_{\mathcal{C}_\xF((s,-T)\times(0,\infty))}\\
&\quad+c||h||_{\mathcal{C}_\xF((s,-T)\times(0,\infty))}
\quad \forall\;t\in [s,-T]\;\text{and}\; i=1,...,k.
\eal
\label{cilemma}
\end{lemma}
\begin{proof}
By change of variables with  $r=(y^{2}+1)^{\frac{1}{2(N-1)}}$, we have

\bal
\int_0^\infty |w'(r(y)-\rho_j(t))|^2r(y)dy&=(N-1)\int_1^\infty |w'(r-\rho_j(t))|^2\frac{r^{2(N-1)}}{\sqrt{r^{2(N-1)}-1}}dr\\
&=(N-1)\xr_j(t)^{N-1}\int_{-\infty}^\infty |w'(r)|^2 dr+O(\xr_j(t)^{N-2}).
\eal
For $i\neq j,$ we obtain

\bal
\int_0^\infty w'(r(y)-\rho_i(t))w'(r(y)-\rho_j(t))r(y)dy&= (N-1)\int_1^\infty w'(r-\rho_i(t))w'(r-\rho_j(t))\frac{r^{2(N-1)}}{\sqrt{r^{2(N-1)-}1}}dr\\
&\leq C (\rho_j(t))^{N-1}|\eta(t)|\left(\frac{|\log|t||}{t}\right)^{|j-i|}
\eal
Thus the system is nearly diagonal and we can solve it for $T$ big enough.

For any $\ell\in\BBN\cup\{0\},$ we can easily prove that
\ba
\int_{1}^\infty r^{\ell}\xF(t,r)dr\leq C\sum_{i=1}^{k}(\rho_i(t))^\ell\left(\frac{\log|t|}{|t|}\right)^\frac{\xs}{2\sqrt{2}}\label{****}
\ea
and
\ba
\int_{1}^\infty r^{\ell}\xF(t,r)w'(r-\rho_j(t))dx \leq C\sum_{i=1}^{k}(\rho_j(t))^\ell\left(\frac{\log|t|}{|t|}\right)^{\frac{\xs}{\sqrt{2}}}.\label{*****}
\ea

Now,

\bal
(N-1)\int_0^\infty\psi(t,y)\left(w''(r(y)-\xr_j(t))y\right)_y dy&=(N-1)\int_0^\infty \psi(t,y)w''(r(y)-\xr_j(t))dy\\
&+(N-1)\int_0^\infty \psi(t,y)w'''(r(y)-\xr_j(t))r'(y)ydy.
\eal

On one hand we have that

\bal
&\bigg|\int_0^\infty \psi(t,y)w''(r(y)-\xr_i(t))dy\bigg|=(N-1)\bigg|\int_1^\infty \psi(t,y(r))w''(r-\xr_i(t))\frac{r^{2N-3}}{\sqrt{r^{2(N-1)}-1}}dr\bigg|\\
&\quad\leq C\left|\left|\psi\right|\right|_{\mathcal{C}_\xF((s,-T)\times(0,\infty))}\int_{1}^\infty \xF(t,r)w'(r-\rho_j(t))\frac{r^{2N-3}}{\sqrt{r^{2(N-1)}-1}}dr\\
&\quad\leq C\left|\left|\psi\right|\right|_{\mathcal{C}_\xF((s,-T)\times(0,\infty))}
\sum_{j=1}^{k}(\rho_j(t))^{N-2}\left(\frac{\log|t|}{|t|}\right)^{\frac{\xs}{\sqrt{2}}}
\eal
On the other hand

\bal
&(N-1)\int_0^\infty \psi(t,y)w'''(r(y)-\xr_j(t))r'(y)ydy+\int_0^\infty f'(z(t,r(y)))\psi(t,y)w'(r(y)-\xr_j(t))r(y)dy\\
&\quad=\int_0^\infty \bigg(f'(z(t,r(y)))-\frac{y^2f'(w(r(y)-\xr_j(t)))}{y^2+1}\bigg)\psi(t,y)w'(r(y)-\xr_j(t))r(y)dy\\
&\quad\leq C\left|\left|\psi\right|\right|_{\mathcal{C}_\xF((s,-T)\times(0,\infty))}
\sum_{j=1}^{k}(\rho_j^0(t))^{-N+1}\left(\frac{\log|t|}{|t|}\right)^{\frac{\xs}{\sqrt{2}}}\\
&\qquad+\int_0^\infty \big(f'(z(t,r(y)))-f'(w(r(y)-\xr_j(t)))\big)\psi(t,y)w'(r(y)-\xr_j(t))r(y)dy.
\eal
In view of the proof of \eqref{213}, we obtain

\bal
&\bigg|\int_0^\infty \big(f'(z(t,r(y)))-f'(w(r(y)-\xr_j(t)))\big)\psi(t,y)w'(r(y)-\xr_j(t))r(y)dy\bigg|\\
&\leq C\left|\left|\psi\right|\right|_{\mathcal{C}_\xF((s,-T)\times(0,\infty))}\\
&\qquad\int_{1-\xr_j(t)}^\infty\big(f'(z(t,r+\xr_j(t)))-f'(w(r))\big)\xF(t,r+\xr_j(t))w'(r) \frac{(r+\xr_j(t))^{2(N-1)}}{\sqrt{(r+\xr_j(t))^{2(N-1)}-1}}dr\\
&\leq C\left|\left|\psi\right|\right|_{\mathcal{C}_\xF((s,-T)\times(0,\infty))}
\sum_{j=1}^{k}(\rho_j^0(t))^{N-1}\left(\frac{\log|t|}{|t|}\right)^{\frac{\sqrt{2}+\xs}{2\sqrt{2}}}.
\eal

By assumptions on $\rho$ we have

\begin{align}
\left|\rho_j'(t)\right|+\left|\frac{1}{r+\rho_j(t)}\right|\leq \frac{C}{\sqrt{|t|}},\label{xi'}
\end{align}
Combining all above, we can easily show the first inequality of the Lemma.

The second inequality is a consequence of \eqref{mainbound} and \eqref{ciest}.

The proof of Lemma is complete.
\end{proof}

\begin{proof}[\textbf{Proof of Lemma \ref{cilemma*}}.]

We will prove that there exists a unique solution of the problem (\ref{proci}) by using a fix point argument.

Let
\bal
X^s=\{\psi:\;||\psi||_{\mathcal{C}_\xF((s,s+1)\times(0,\infty)}<\infty\}.
\eal

We consider the operator $A^s: X^s\rightarrow X^s$ given by
\bal
A^s(\psi)=T^s(h-D(\psi)),
\eal
where $T^s(g)$ denotes the solution to (\ref{mainpro*}) and $D(\psi)=\sum_{j=1}^k d_j(t)w'(x-\rho_j(t)).$ By \eqref{normest}, we have
\be
||A^s(\psi)||_{\mathcal{C}_\xF((s,s+1)\times(0,\infty))}\leq c_1||h-D(\psi)||_{\mathcal{C}_\xF((s,s+1)\times(0,\infty))}\label{lem1}
\ee
for some uniform constant $c_1>0.$

We will show that the map $A^s$ defines a contraction mapping and we will apply the fixed point theorem to it.
To this end, let $c$ be the constant in Lemma \ref{cilemma}. We set
\bal
c_2=\max(c,1)c_1||h||_{\mathcal{C}_\xF((s,-T)\times(0,\infty))}
\eal
and
\bal
X^s_c=\{\psi:\;||\psi||_{C_\xF((s,s+1)\times(0,\infty))}<2c_2\}.
\eal

We claim that $A^s(X^s_c)\subset X^s_c,$ indeed by \eqref{lem1} we have
\bal
\nonumber
&||A^s(\psi)||_{\mathcal{C}_\xF((s,s+1)\times(0,\infty))}\leq c_1||h-D(\psi)||_{\mathcal{C}_\xF((s,s+1)\times(0,\infty))}\\
&\leq c_1\left(||h||_{\mathcal{C}_\xF((s,-T)\times(0,\infty))}+
c_1||D(\psi)||_{\mathcal{C}_\xF((s,s+1)\times(0,\infty))}\right)\\
&\leq
\frac{c_1c}{\sqrt{|s+1|}}\left(||\psi||_{\mathcal{C}_\xF((s,s+1)\times(0,\infty))}\right)+c_2\\
&\leq c_2+c_2,
\eal
where in the above inequalities we have used Lemma \ref{cilemma} and we have chosen $|s|$ big enough. Next we show that $A^s$ defines a contraction map. Indeed, since $D(\psi)$ is linear in $\psi$ we have
\bal
&||A^s(\psi_1)-A^s(\psi_2)||_{\mathcal{C}_\xF((s,s+1)\times(0,\infty))}\\
&\leq c_1||D(\psi_1)-D(\psi_2)||_{\mathcal{C}_\xF((s,s+1)\times(0,\infty))}=
c_1 ||D(\psi_1-\psi_2)||_{\mathcal{C}_\xF((s,s+1)\times(0,\infty))}\\
&\leq\frac{cc_1}{\sqrt{|s+1|}}||(\psi_1-\psi_2)||_{\mathcal{C}_\xF((s,s+1)\times(0,\infty))}\\
&\leq\frac{1}{2}||(\psi_1-\psi_2)||_{\mathcal{C}_\xF((s,s+1)\times(0,\infty))}.
\eal
Combining all above, we have by fixed point theorem that there exists a $\psi^s\in X^s$ so that $A^s(\psi^s)=\psi^s,$ meaning that the equation (\ref{proci}) has a solution $\psi^s,$ for $-T=s+1.$

We claim that $\psi^s(t,x)$ can be extended to a solution on $(s,-T_0]\times(0,\infty),$ still satisfies the orthogonality condition (\ref{orthcond}) and the a priori estimate. To this end, assume that our solution $\psi(t,\cdot)$ exists for $s\leq t\leq -T,$ where $T>T_0$ is the maximal time of the existence. Since $\psi^s$ satisfies the orthogonality condition (\ref{orthcond}), we have by \eqref{estfix}
\bal
||\psi^s||_{\mathcal{C}_\xF((s,-T)\times(0,\infty))}\leq C||h-D(\psi)||_{\mathcal{C}_\xF((s,-T)\times(0,\infty))}.
\eal
Thus if we choose $T_0$ big enough, we have by Lemma \ref{cilemma} that
\bal
||\psi^s||_{\mathcal{C}_\xF((s,-T)\times(0,\infty))}&\leq C||h||_{\mathcal{C}_\xF((s,-T)\times(0,\infty))}\\
&\leq C||h||_{\mathcal{C}_\xF((s,-T_0)\times(0,\infty))}
\eal
It follows that $\psi^s$ can be extended past time $-T,$ unless $T=T_0.$ Moreover, (\ref{estfixmef}) is satisfied as well and $\psi^s$ also satisfies the orthogonality condition.
\end{proof}

\begin{proposition}  \label{prop1} Let $T_0>2$ be the constant in Lemma \ref{cilemma*}. For each $h\in\mathcal{C}_\xF((-\infty ,-T_0)\times\mathbb{R}),$ there exists a solution $\psi\in \mathcal{C}_\xF((-\infty ,-T_0)\times(0,\infty))$ satisfying \eqref{orthcon} and

\ba&\BAL
\psi_t&=\frac{(N-1)^2}{(y^2+1)^{\frac{1}{2(N-1)}}}\frac{d}{dy}\left((y^2+1)^{\frac{2N-3}{2(N-1)}}\psi_y\right)+f'(z(t,r(y)))\psi &&\\
&\qquad\qquad+h(t,y)-\sum_{j=1}^k d_i(t)w'(r-\rho_j(t)),\quad &&\mathrm{in}\;\;\;(-\infty,-T_0]\times(0,\infty),
\EAL
\ea
where $d_i$ satisfies \eqref{ci(t)} with $s=-\infty.$ Furthermore, the operator, given by $\psi = A(h),$ is
a linear operator of $h$ and satisfies the estimate
\be
||A(h)||_{\mathcal{C}_\xF((-\infty,t)\times(0,\infty))}\leq C||h||_{\mathcal{C}_\xF((-\infty,t)\times(0,\infty))}\qquad\forall t\leq -T_0,
\label{estfix**}
\ee
where $C$ is a uniform constant.
\end{proposition}
\begin{proof}
 Take a sequence $s_j\rightarrow-\infty$ and  $\psi_j=\psi^{s_j}$ where $\psi^{s_j}$ is the function   (\ref{proci}) with $s=s_j.$ Then by (\ref{estfix}), we can find a subsequence $\{\psi_j\}$ and $\psi$ such that $\psi_j\rightarrow\psi$ locally uniformly in $(-\infty,-T_0)\times(0,\infty).$

Using \eqref{estfixmef} and standard parabolic theory, we may deduce that $\psi$ is a solution of (\ref{proci}) and satisfies (\ref{estfix**}).
The proof is complete.
\end{proof}

\section{The nonlinear problem}
Going back to the nonlinear problem, function $\psi$  is a solution of (\ref{mainpro}) if and only if  $\psi\in C_\xF((-\infty,-T_0)\times(0,\infty))$ solves the fixed point problem
\be
  \psi= H(\psi ) \label{2.14}
\ee
where
\bal
H(\psi ) := A(\overline{E}(\psi)),
\eal
$A$ is the operator in Proposition \ref{prop1} and
\bal
\overline{E}(\psi)=E+N(\psi)-\sum_{i=1}^k d_i(t)w'(x-\rho_i(t)).
\eal
Let $T_0>2,$ we define
\bal
\xL=\left\{h\in C^1(-\infty,-T_0]:\;\sup_{t\leq -T_0}|h(t)|+\sup_{t\leq -T_0}\left(\frac{|t|}{\log|t|}|h'(t)|\right)<1\right\}
\eal
and
\bal
||h||_\xL=\sup_{t\leq -T_0}(|h(t)|)+\sup_{t\leq -T_0}\left(\frac{|t|}{\log|t|}|h'(t)|\right).
\eal
The main goal in this section is to prove the following Proposition.
\begin{proposition}
Let $\xs<\sqrt{2}$ and $\xn=\frac{\sqrt{2}-\xs}{2\sqrt{2}}$. There exists number $T_0> 0,$ depending only on $\xs$ such that for any given functions $h$ in $\xL,$ there is a solution  $ \psi= 	 \Psi(h)$ of (\ref{2.14}), with respect $\rho=\rho^0+h.$ The solution $\psi$
satisfies the orthogonality condition \eqref{orthcon}.
Moreover, the following estimate holds
\be
||\Psi(h_1)-\Psi(h_2)||_{\mathcal{C}_\xF((-\infty,-T_0)\times(0,\infty))}\leq C\left(\frac{\log T_0}{T_0}\right)^\xn||h_1-h_2||_\xL,\label{diafora}
\ee
where $C$ is a universal constant.\label{mainproposition}
\end{proposition}

To prove Proposition \ref{mainproposition} we need to prove some lemmas first. Set
\bal
X_{T_0}=\{\psi:\;||\psi||_{\mathcal{C}_\xF((-\infty,-T_0)\times(0,\infty))}<2C_0\left(\frac{\log T_0}{T_0}\right)^\xn\},
\eal
for some fixed constant $C_0.$

We denote by $N(\psi,h)$ the function $N(\psi)$ in \eqref{mainlem} with respect $\psi$ and $\rho=\rho^0+h.$
 Also we denote by $z_i$ the respective function in \eqref{z} with respect to $\rho=\rho_i=\rho^0+h_i,$ $i=1,2.$

\begin{lemma}
Let $h_1,\;h_2\in \xL$ and $\psi_1,\;\psi_2\in X_{T_0}.$
Then there exists a constant $C=C(C_0)$ such that
\begin{align*}
||N(\psi_1,h_1)&-N(\psi_2,h_2)||_{\mathcal{C}_\xF((-\infty,-T_0)\times(0,\infty))}\\ &
\leq C\left(\frac{\log T_0}{T_0}\right)^\xn\left(||\psi_1-\psi_2||_{\mathcal{C}_\xF((-\infty,-T_0)\times(0,\infty))}+\left(\frac{\log T_0}{T_0}\right)^\xn||h_1-h_2||_{\xL}\right).
\end{align*}\label{dia1}
\end{lemma}
\begin{proof}
First we will prove that there exists constant $C>0$ depending only on $C_0$ such that
\be
||N(\psi_1,h_1)-N(\psi_2,h_1)||_{\mathcal{C}_\xF((-\infty,-T_0)\times(0,\infty))}\leq C\left(\frac{\log T_0}{T_0}\right)^\xn||\psi_1-\psi_2||_{\mathcal{C}_\xF((-\infty,-T_0)\times(0,\infty))}.\label{n1}
\ee
By straightforward calculation we can easily show that
\ba\label{1}
|N(\psi_1,h_1)-N(\psi_2,h_1)|\leq C\left(\frac{\log T_0}{T_0}\right)^\xn|\psi_1-\psi_2|(\xF+\xF^2),
\ea
where the constant $C$ depend on $C_0.$ This implies \eqref{n1}.

Next we will prove that
\be
||N(\psi_2,h_1)-N(\psi_2,h_2)||_{\mathcal{C}_\xF((-\infty,-T_0)\times(0,\infty))}\leq C\left(\frac{\log T_0}{T_0}\right)^\xn||h_1-h_2||_{\xL}.\label{n2}
\ee
where the constant $C$ depends on $C_0.$

By straightforward calculations we have
\ba\BAL
|N(\psi_2,h_1)-N(\psi_2,h_2)|&=|-(z_1+\psi_2)^3+z_1^3+3z_1^2\psi_2+(z_2+\psi_2)^3-z_2^3-3z_2^2\psi_2|\\
&\leq C\left(\frac{\log T_0}{T_0}\right)^{2\xn}|h_1-h_2|\xF^2,\label{n2*}
\EAL
\ea
which implies \eqref{n2}. By \eqref{n1} and \eqref{n2}, the desired result follows.
\end{proof}
We denote by $E(h)$ the function $E$ in \eqref{mainlem} with respect to  $\rho=\rho^0+h.$
\begin{lemma}
Let $h_1,\;h_2\in \xL.$
Then there exists constant $C=C(C_0)$ such that
\be
||E(h_1)-E(h_2)||_{\mathcal{C}_\xF((-\infty,-T_0)\times(0,\infty))}\leq C\left(\frac{\log T_0}{T_0}\right)^\xn||h_1-h_2||_{\xL}
\ee\label{dia2}
\end{lemma}
\begin{proof} In view of the proof of Lemma \ref{remark}, we can easily show that there exists a constant $C$ such that

\ba\label{2}
|E(h_1)-E(h_2)|\leq C\left(\frac{\log |t|}{|t|}\right)^\xn \left(|h_1-h_2|+ \frac{|t|}{\log|t|}|h_1'-h_2'|\right)\xF.
\ea

This implies the desired result.
\end{proof}

\begin{lemma}
Let $h_1,\;h_2\in \xL,$ $\psi_1,\;\psi_2,\;\psi\in X.$  Also let $D(\psi,h,t)=(d_1(t),...,d_k(t))$ satisfy \eqref{ci(t)} with respect to $\psi$ and $\rho=\rho^0+h.$ Then
\begin{align}\nonumber
|D(\psi_1,h_1,t)-D(\psi_2,h_2,t)|&\leq C\left(\frac{|\log|t|}{|t|}\right)^{\frac{1}{2}+\frac{\xs}{\sqrt{2}}}||\psi_1-\psi_2||_{\mathcal{C}_\xF((-\infty,t)\times(0,\infty))}\\ &+C\left(\frac{\log|t|}{|t|}\right)^{\xn+\frac{\xs}{\sqrt{2}}}||h_1-h_2||_{\xL(-\infty,t)}
\end{align}
for some positive constant $C$ which depends only on $C_0.$
\label{dia3}
\end{lemma}
\begin{proof}
The result follows by \eqref{1}, \eqref{n2*}, \eqref{2} and Lemma \ref{cilemma}.
\end{proof}

\medskip
\begin{proof}[\textbf{Proof of Proposition \ref{mainproposition}}]
a)
We consider the operator $H: \mathcal{C}_\xF((-\infty,-T_0)\times(0,\infty))\rightarrow \mathcal{C}_\xF((-\infty,-T_0)\times(0,\infty)),$
where $H(\psi)$ denotes the solution to (\ref{2.14}).
We will show that the map $H$ defines a contraction mapping and we will apply the fixed point theorem to it.
First we note by Lemma \ref{remark} and Proposition \ref{prop1} that

\bal
||H(0)||_{\mathcal{C}_\xF((-\infty,-T_0)\times(0,\infty))}\leq  C_0\left(\frac{\log T_0}{T_0}\right)^\xn.
\eal
and by Proposition \ref{prop1} and Lemma \ref{dia1}
\bal
||H(\psi_1)-H(\psi_2)||_{\mathcal{C}_\xF((-\infty,-T_0)\times(0,\infty))}\leq  C\left(\frac{\log T_0}{T_0}\right)^\xn
\left(||\psi_1-\psi_2||_{\mathcal{C}_\xF((-\infty,-T_0)\times(0,\infty))}\right),
\eal
providing
\bal
||\psi_i||_{\mathcal{C}_\xF((-\infty,-T_0)\times(0,\infty))}\leq2 C_0\left(\frac{\log T_0}{T_0}\right)^\xn.
\eal
Thus if we choose $T_0$ big enough we can apply the fix point Theorem in
\bal
X_{T_0}=\{\psi:\;||\psi||_{\mathcal{C}_\xF((-\infty,-T_0)\times(0,\infty))}<2 C_0\left(\frac{\log T_0}{T_0}\right)^\xn\},
\eal
to obtain that there exists $\psi$ such that $H(\psi)=\psi.$

b) For simplicity we set  $\psi^1 = 	\Psi(h_1)$ and
 $\psi^2 = 	\Psi(h_2).$ The estimate will be obtained by applying estimate \eqref{estfix**}. However, because each
 $\psi^i$ satisfies the orthogonality conditions (\ref{orthcon})  with $\rho(t) = \rho^i(t) := \rho^0(t)+ h_i(t),$ the
difference  $\psi^1-\psi^2$ doesn't satisfy an exact orthogonality condition. To overcome this technical difficulty
we will consider instead the difference $Y :=  \psi^1-\overline{\psi}^2,$ where

\bal
\overline{\psi}^2=\psi^2-\sum_{i=1}^k\xl_i(t)w'(r(y)-\rho_i^1)
\eal
with
\bal
\sum_{i=1}^k\xl_i(t)\int_{0}^\infty w'(r(y)-\rho_i^1(t))w'(r(y)-\rho_j^1(t))r(y)dy=\int_{0}^\infty \psi^2(t,y)w'(r(y)-\rho_j^1(t)) r(y)dy,
\eal
$j=1,...,k.$
Clearly, $Y$ satisfies the orthogonality conditions (\ref{orthcon}) with $\rho(t) = \rho^1(t).$  In addition, by Lemmas \ref{dia1}, \ref{dia2} and \ref{dia3}
and \eqref{mainbound}, we can easily prove
\ba\BAL
||Y||_{\mathcal{C}_\xF((-\infty,-T_0)\times(0,\infty))}&\leq C \left(\frac{\log T_0}{T_0}\right)^\xn\left(||\psi_1-\psi_2||_{\mathcal{C}_\xF((-\infty,-T_0)\times(0,\infty))}+||h_1-h_2||_{\xL}\right) \\ &+C\left(\frac{\log T_0}{T_0}\right)^\xn\left(\sum_{i=1}^k\sup_{t\in(-\infty,-T_0)}\bigg(\frac{|t|}{\log|t|}\bigg)^{\frac{\xs}{\sqrt{2}}}
(|\xl_i(t)|+|\xl_i'(t)|)\right).\EAL\label{pao0}
\ea
Now, by orthogonality conditions \eqref{orthcon} and estimate \eqref{*****}, we have
\ba\BAL
\left|\int_{0}^\infty \psi^2(t,y)w'(r(y)-\rho_j^1(t))r(y)dy\right|&=\left|\int_{0}^\infty \psi^2(t,y)(w'(r(y)-\rho_j^1(t))-w'(r(y)-\rho_j^2(t)))r(y)dy\right|\\
&\leq C \left(\frac{\log T_0}{T_0}\right)^\xn\bigg(\frac{|t|}{\log|t|}\bigg)^{-\frac{\xs}{\sqrt{2}}}||h_1-h_2||_{\xL}\sum_{i=1}^k(\rho_i^0)^{N-1}.\label{pao1}
\EAL
\ea

Now
\begin{align}\nonumber
&\left|\frac{d\int_{0}^\infty \psi^2(t,r(y))w'(r(y)-\rho_j^1(t))r(y)dy}{dt}\right|\\
&=\left|\frac{d\int_{0}^\infty \psi^2(t,r(y))(w'(r(y)-\rho_j^1(t))-w'(r(y)-\rho_j^2(t)))r(y)dy}{dt}\right|.\label{pao2}
\end{align}

Set $W(t,y):= w'(r(y)-\rho_j^1(t))-w'(r(y)-\rho_j^2(t)),$ then

\bal
&\int_{0}^\infty \psi^2_t(t,y)W(t,y)r(y)dy=(N-1)^2\int_0^\infty\psi^2(t,y)\frac{d}{dy}\left((y^2+1)^{\frac{2N-3}{2(N-1)}}W_y\right)dy\\
&+\int_{0}^\infty \psi^2(t,y)f'(z(t,r(y)))W(t,y)r(y)dy+\int_{0}^\infty (E(h_2)+N(\psi^2,h_2) )W(t,y)r(y)dy.
\eal
This implies that
\ba
\left|\int_{0}^\infty \psi^2_t(t,y)W(t,y)r(y)dy\right|
\leq C \left(\frac{\log T_0}{T_0}\right)^\xn\bigg(\frac{|t|}{\log|t|}\bigg)^{-\frac{\xs}{\sqrt{2}}}||h_1-h_2||_{\xL}\sum_{i=1}^k(\rho_i^0)^{N-1}.\label{pao3}
\ea
By \eqref{pao1}, \eqref{pao2}, \eqref{pao3} and definitions of $\xl_i$ we have that
\bal
 |\xl_i(t)|+|\xl_i'(t)|\leq C \left(\frac{\log T_0}{T_0}\right)^\xn\bigg(\frac{|t|}{\log|t|}\bigg)^{-\frac{\xs}{\sqrt{2}}}||h_1-h_2||_{\xL}.
\eal
Combining all above we have that
\bal
||Y||_{\mathcal{C}_\xF((-\infty,-T_0)\times(0,\infty))}\leq C \left(\frac{\log T_0}{T_0}\right)^\xn\left(||\psi_1-\psi_2||_{\mathcal{C}_\xF((-\infty,-T_0)\times(0,\infty))}+||h_1-h_2||_{\xL}\right)
\eal
But
\bal
||\psi_1-\psi_2||_{\mathcal{C}_\xF((-\infty,-T_0)\times(0,\infty))}&\leq ||Y||_{\mathcal{C}_\xF((-\infty,-T_0)\times(0,\infty))}+C\left(\sum_{i=1}^k\sup_{t\in(-\infty,-T_0)}
\bigg(\frac{|t|}{\log|t|}\bigg)^{\frac{\xs}{\sqrt{2}}}|\xl_i(t)|\right)\\
&\leq  C \left(\frac{\log T_0}{T_0}\right)^\xn\left(||\psi_1-\psi_2||_{\mathcal{C}_\xF((-\infty,-T_0)\times(0,\infty))}+||h_1-h_2||_{\xL}\right),
\eal
and the proof of inequality \eqref{diafora} follows, if we choose $T_0$ large enough.
\end{proof}

\section{the choice of $\rho_i$}\label{xiint}
Let $T_0$ be large enough, $\frac{\sqrt{2}}{2}<\xs<\sqrt{2}$ and $\psi\in\mathcal{C}_\xF((-\infty,-T_0)\times(0,\infty))$ be the solution of problem (\ref{mainpro}).
We want to find $\rho_i$ such that $d_i=0$ in \eqref{ole} for any $i=1,...,k.$ We recall here that 
\bal
\xL=\{h\in C^1(-\infty,-T_0]:\;\sup_{t\leq -T_0}|h(t)|+\sup_{t\leq -T_0}\frac{|t|}{\log|t|}|h'(t)|<1\}
\eal
and
\bal
||h||_\xL=\sup_{t\leq -T_0}(|h(t)|)+\sup_{t\leq -T_0}(\frac{|t|}{\log|t|}|h'(t)|)
\eal
for some $T_0>2.$ In addition, we assume that  $\rho=\rho^0+h$ for some $h\in\xL.$

 Let $1<j<k,$ then we have that
\bal
&\int_0^\infty \bigg(f(z(t,r(y)))-\sum_{i=1}^{k}(-1)^{i+1}f(w(r(y)-\rho_i(t)))\bigg)w'(r(y)-\rho_j(t))r(y)dy\\
&=\int_{1-\rho_j(t)}^\infty \bigg(f(z(t,x+\rho_j(t)))-\sum_{i=1}^{k}(-1)^{i+1}f(w(x+\rho_j(t)-\rho_i(t)))\bigg)\\
&\qquad\qquad w'(x)
\frac{(x+\xr_j(t))^{2(N-1)}}{\sqrt{(x+\xr_j(t))^{2(N-1)}-1}}dx\\
&=\xr_j(t)^{N-1}\int_{-\infty}^\infty\left(f(z(t,x+\rho_j(t)))-\sum_{i=1}^{k}(-1)^{i+1}f(w(x+\rho_j(t)-\rho_i(t)))\right)w'(x)dx +F_0(\xr)
\eal
Let $\xe\in (0,1).$ By \eqref{213}, there exists a positive constant $C=C(\xe,T_0,k)$ such that 
\bal
|F_0(\xr)|\leq C(\xe,T_0,k)\xr_j(t)^{N-2}\bigg(\frac{|t|}{\log|t|}\bigg)^{-\xe}.
\eal

Next, for simplicity we assume that $j$ is even and we write $z=g+g_1+g_2-w(x),$ where

\begin{align*}
g&=\sum_{i=1}^{j-2}(-1)^{i+1}\left(w(x+\rho_j(t)-\rho_{i}(t))-1\right)
\\&+\sum_{i=j+2}^{k}(-1)^{i+1}\left(w(x+\rho_j(t)-\rho_{i}(t))+1\right),
\end{align*}
\bal
g_1=w(x+\rho_j-\rho_{j-1})-1,
\eal
and
\bal
g_2=w(x+\rho_j-\rho_{j+1})+1.
\eal

By straightforward calculations we have

\begin{align}\nonumber
I&=\int_{-\infty}^\infty \left(f(z(t,x-\rho_j(t)))-\sum_{i=1}^{k}(-1)^{i+1}f(w(x+\rho_j(t)-\rho_i(t)))\right)w'(x)dx\\ \nonumber
&=3\int_{-\infty}^\infty (g_1+g_2)(1-w^2(x))w'(x)dx+3\int_{-\infty}^\infty g_1^2(1+w(x))w'(x)dx \\ \nonumber
&+3\int_{-\infty}^\infty g_2^2(w(x)-1)w'(x)dx
+\int_{-\infty}^\infty G_1(t,x)w'(x)dx,
\end{align}
where $G_1(t,x)=O(g)+O(g_1g_2).$ Hence,

\bal
I=3\int_{-\infty}^\infty (g_1+g_2)(1-w^2(x))w'(x)dx+F_3(\xr),
\eal
where

\bal
|F_3(\xr)|\leq C(\xe,T_0,k)\xr_j(t)^{N-2}\bigg(\frac{|t|}{\log|t|}\bigg)^{-\xe-1}.
\eal

Now,

\bal
3\int_{-\infty}^\infty (g_1+g_2)(1-w^2(x))w'(x)dx&=-6e^{-\sqrt{2}(\rho_{j}-\rho_{j-1})}\int_{-\infty}^\infty e^{-\sqrt{2}x}(1-w^2(x))w'(x)dx\\
&+6e^{-\sqrt{2}(\rho_{j}-\rho_{j+1})}\int_{-\infty}^\infty e^{\sqrt{2}x}(1-w^2(x))w'(x)dx+F_4(\xr),
\eal
where

\bal
|F_4(\xr)|\leq C(\xe,T_0,k)\bigg(\frac{|t|}{\log|t|}\bigg)^{-\xe-1}.
\eal

Treating the other terms in a similar way, we may show that if $d_i=0$ in \eqref{ole} for any $i=1,...,k,$ then $\xr=(\xr_1,...,\xr_k)$ satisfies the following ODE

\be
\rho_j'+\frac{n-1}{\rho_j}-\xb e^{-\sqrt{2}(\rho_{j+1}-\rho_j)}+\xb e^{-\sqrt{2}(\rho_{j}-\rho_{j-1})}=G_j(\rho',\rho),\qquad \forall j=1,2,...,k\;\text{and}\;t\in(-\infty,-T_0],\label{ode*}
\ee
for some functions $G_j$ $(j=1,...,k)$ such that 

\bal
\max_{j\in\{1,...,k\}}|G_j(\xr',\xr)|\leq \frac{C(k,N,\xs)}{|t|}\quad\forall t\leq -T_0.
\eal
Here, we have set $\rho_{k+1}=\infty,$ $\rho_0=-\infty$ and
\be \label{beta} \xb=\frac{6\int_{\mathbb{R}}e^\frac{2x}{\sqrt{2}}(1-w^2(x))w'(x)dx}{\int_{\mathbb{R}}(w'(x))^2dx}.\ee

Let
\bal
\overline{\mathbf{F}}(h',h)=\mathbf{G}(\rho',\rho),
\eal
where $\rho=\rho^0+h.$ By proceeding as above and using \eqref{diafora}, we may deduce the following result.

\begin{proposition}
Let $\frac{\sqrt{2}}{2}<\xs<\sqrt{2}$ and $h,\;h_1,\;h_2\in \xL.$ Then there exists a constant $C=C(\xs,n,k)>0$ such that
\bal
|\overline{\mathbf{F}}(h',h)|\leq \frac{C}{|t|},
\eal
and
\bal
|\overline{\mathbf{F}}(h'_1,h_2)-\overline{\mathbf{F}}(h'_1,h_2)|\leq C\left(\frac{\log|t|}{|t|}\right)^{\frac{1}{2}+\frac{\xs}{\sqrt{2}}}||h_1-h_2||_{\xL}.
\eal
\label{remark2}
\end{proposition}

Finally, following the arguments in \cite[Subections 5.1 and 5.2]{gklpino1}, we may show the following result. 

\begin{proposition}
Let $T_0>2.$ Then there exist a solution $\rho=(\rho_1,...,\rho_k)$ to equation \eqref{ode*}. In addition, there holds  

\bal
\rho(t) = \rho^0(t)  + h(t)
\eal
for some function $h=(h_1,...,h_k)$ satisfying
  
\ba
\sup_{t\leq -T_0}|\log|t||\,|{h}(t)|+\sup_{t\leq -T_0}\frac{|t|}{\log|t|}|{h}'(t)|<1\quad\text{for}\;j=1,...,k.
\ea
\end{proposition}


\begin{thebibliography}{99}


\bibitem{AdW} O. Agudelo, M. del Pino, J. Wei, \emph{Solutions with multiple catenoidal ends to the Allen-Cahn
equation in $\BBR^3$.} J. Math. Pures Appl. (9) 103 (2015), no. 1, 142–218.

\bibitem{ac} S. M. Allen and J. W. Cahn, {\em A microscopic theory for antiphase boundary motion and its application to antiphase domain coarsening,} Acta. Metall., 27 (1979), 1084-–1095.

\bibitem{bk} {L. Bronsard, R. V. Kohn, \emph{Motion by mean curvature as the singular limit of Ginzburg-Landau dynamics.} J. Differential Equations 90 (1991) 211-237.}
    
    
\bibitem{carrpego}  J. Carr, J., R.L. Pego, {\em
Invariant manifolds for metastable patterns in $u_t=\varepsilon^2 u_{xx}-f(u)$.} Proc. Roy. Soc. Edinburgh Sect. A 116 (1990), no. 1-2, 133-160.    

\bibitem{chen}{X. Chen,\emph{ Generation and propagation of interfaces for reaction diffusion equations,} J. Differential Equations 96 (1992) 116--141.}


\bibitem{5authors}
    X. Chen, J.-S. Guo, F. Hamel, H. Ninomiya, J. Roquejoffre, {\em
    Traveling waves with paraboloid like interfaces for balanced bistable dynamics.} Ann. Inst. H. Poincare Anal. Non Lineaire 24 (2007), no. 3, 369–-393.


\bibitem{gklpino1} M. del Pino, K. Gkikas, \emph{Ancient shrinking spherical interfaces in the Allen–Cahn flow.} In: Ann. Inst. H.
Poincar\'e Anal. NonLin\'eaire, vol. 35, No. 1, 187-215 (2018).

\bibitem{gklpino2} M. del Pino, K. Gkikas, \emph{Ancient multiple-layer solutions to the Allen–Cahn equation.} Proc. R. Soc.
Edinb. Sect. A 148(6), 1165-1199 (2018).



\bibitem{dkwcpam}
    M. del Pino, M. Kowalczyk, J. Wei, {\em
    Traveling waves with multiple and nonconvex fronts for a bistable semilinear parabolic equation.} Comm. Pure Appl. Math. 66 (2013), no. 4, 481–-547.


\bibitem{dkw3}
M. del Pino, M. Kowalczyk, J. Wei, J. Yang  {\em
J. Interface foliation near minimal submanifolds in Riemannian manifolds with positive Ricci curvature.} Geom. Funct. Anal. 20 (2010), no. 4, 918–-957.


\bibitem{dkw0} M. del Pino, M. Kowalczyk, J. Wei,   {\em The Toda system and clustering interfaces in the Allen-Cahn equation. } Arch. Ration. Mech. Anal. 190 (2008), no. 1, 141–-187.


\bibitem{dkpw}
M. del Pino, M. Kowalczyk, F. Pacard, J. Wei, {\em  Multiple-end solutions to the Allen-Cahn equation in $\BBR^2$.}
 J. Funct. Anal. 258 (2010), no. 2, 458–-503.

\bibitem{evans} L.C. Evans, H.M. Soner, P.E. Souganidis, \emph{Phase transitions and generalized motion by mean curvature,} Comm. Pure Appl. Math. 45 (1992)
1097–1123.


\bibitem{fuscohale}
G. Fusco, J.K. Hale, {\em
Slow-motion manifolds, dormant instability, and singular perturbations.}
J. Dynam. Differential Equations 1 (1989), no. 1, 75-94.





\bibitem{ninomiya}
H. Ninomiya and M. Taniguchi, {\em Existence and global stability of traveling curved fronts in the Allen-Cahn equations}, J. Differential Equations, 213 (2005), 204-233.


\end{thebibliography}
\end{document}